\documentclass[preprint,12pt]{elsarticle}

\usepackage{amssymb}
\usepackage{amssymb,amscd,amsthm, verbatim,amsmath,color,fancyhdr, mathrsfs}
\usepackage{graphicx}
\usepackage{turnstile}
\usepackage[colorinlistoftodos]{todonotes}
\usepackage{tikz-cd}
\usepackage{mathtools}

\usepackage[utf8]{inputenc}

\usepackage[letterpaper, left=3.75cm, right=3.75cm, top=2.5cm,
bottom=2.5cm,dvips]{geometry}
\usepackage{stmaryrd}

\begin{document}

\begin{frontmatter}



\newtheorem{thm}{Theorem}[section]
\newtheorem{prop}[thm]{Proposition}
\newtheorem{lem}[thm]{Lemma}
\newtheorem{rem}[thm]{Remark}
\newtheorem{cor}[thm]{Corollary}
\newtheorem{exr}[thm]{Exercise}
\newtheorem{ex}[thm]{Example}
\newtheorem{defn}[thm]{Definition}

\title{Soft Connectedness, Soft Path Connectedness and the Category of Soft Topological Groups}


\author[inst1]{Nazmiye Alemdar}
\ead{nakari@erciyes.edu.tr}

\author[inst2]{H\"{u}rmet Fulya Ak{\i}z\corref{correspondingauthor}}
\cortext[correspondingauthor]{Corresponding author: H\"{u}rmet Fulya Ak{\i}z}
\ead{hfulya@gmail.com}

\author[inst3]{Halim Ayaz}
\ead{mrckl_halim@hotmail.com}
\affiliation[inst1]{organization={Department of  Mathematics, Faculty of Science},
            addressline={Erciyes University},
            postcode={38039},
            state={Kayseri},
            country={Turkey}}
 \affiliation[inst2]{organization={Department of  Mathematics, Faculty of Art and Science},
            addressline={Bozok University},
            city={Yozgat},
            postcode={66100},
             country={Turkey}}
\affiliation[inst3]{organization={Graduate School of Natural and Applied Sciences},
            addressline={Erciyes University},
            city={Kayseri},
            postcode={38039},
            country={Turkey}}
\begin{abstract}
In this study, the soft usual topology compatible with the usual topology of $\mathbb{R}$ is defined, and using its subspace topology on the interval $[0,1]$, the concept of a soft path is introduced. Within this context, the notions of soft connectedness and soft path connectedness are developed, their relationship is analyzed, and it is shown that these properties are preserved under soft continuous mappings. Moreover, the behavior of these concepts within soft topological groups is investigated in detail. Finally, the category of soft topological groups is constructed, its morphisms are identified, and it is shown that this category forms a symmetric monoidal category.
\end{abstract}


\end{frontmatter}

\section{Introduction}

\label{intro}
With the development of fuzzy logic and uncertainty theories, the variety of mathematical structures used to model uncertain data has significantly increased \cite{Zadeh,Pawlak}. In 1999, Molodtsov  introduced the theory of soft sets as a parameter-based approach, extending the applicability of the classical set theory to uncertain environments \cite{Mol}. His aim was to overcome some of the difficulties encountered in traditional approaches and to provide a parameter-oriented solution method for problems in fields such as physics, economics, and engineering. Consequently, this theory has found applications in operations research, game theory, and probability theory \cite{Tripathy}.

Following the emergence of the soft set concept, both topological and algebraic notions on this structure developed rapidly. Maji, Biswas, and Roy, in \cite{Maji}, studied the fundamental properties of soft sets; Akta\c{s} and \c{C}a\v{g}man, in \cite{br1}, introduced the concept of soft groups, while Shabir and Naz \cite{Naz} as well as \c{C}a\v{g}man, Karata\c{s}, and Engino\v{g}lu \cite{Naim} independently defined soft topological spaces. Kharal and Ahmad, in \cite{Kha} proposed the notion of soft mappings, and Ayg\"{u}no\v{g}lu and Ayg\"{u}n \cite{Aygun}, together with Zorlutuna et al. in \cite{Zorlutuna}, introduced soft continuity. Later, Hida, in \cite{Hida}, provided a new and distinct definition of soft continuity and established the notion of soft topological groups. This structure reinterprets classical topological groups within the framework of soft topology, allowing the study of algebraic and topological properties under parameterized uncertainty.The structure of soft topological groups has also been examined in a different way in \cite{Tahat}.

Category theory, first introduced in the 1940s by Samuel Eilenberg \cite{Eilenberg} and Saunders Mac Lane \cite{MacLane}, plays a role in mathematics comparable to that of set theory as a unifying language across nearly all disciplines. More specifically, it identifies similarities across different fields of mathematics and provides a common framework for unification. A category consists of objects and morphisms between them, where composition and identity laws hold. Category theory has been particularly influential in topology, especially in the development of homotopy theory and homology theory. Monoidal categories were independently introduced by B\'{e}nabou \cite{Benabou} as "categories with multiplication" and by Mac Lane \cite{MacLane2} under the name "bicategories" (not to be confused with two-categories). These structures include additional data explaining how objects and morphisms can be combined in parallel.

Monoidal categories extend beyond basic category theory and find significant applications in modeling the multiplicative fragment of intuitionistic linear logic, in establishing the mathematical foundation of topological order in condensed matter physics and in quantum information, quantum field theory, and string theory. Every cartesian monoidal category is also a symmetric monoidal category \cite{Baez}. These structures are further employed in modeling open games in game theory \cite{Tripathy}. Recently, Alemdar and Arslan \cite{AlemdarArslan} constructed the category SGp, consisting of soft groups as objects and soft group homomorphisms as morphisms, and proved that it forms a symmetric monoidal category.

Topological group theory is an important field where algebraic and topological structures are studied together. In classical topological groups, the fundamental group $\pi_1(X, x_0)$ is among the most powerful tools of algebraic topology, often used to investigate topological properties of spaces \cite{Rotman}. If $X$ is a connected topological space with a universal covering and $x_0 \in X$, then for a subgroup $G$ of $\pi_1(X, x_0)$, according to \cite[Theorem~10.42]{Rotman}, there exists a covering map $p : (\widetilde{X_G}, \widetilde{x_0}) \to (X, x_0)$ with characteristic group $G$. Thus, a new structure is obtained that preserves both the group structure and the covering property. This method has been generalized to topological $R$-modules \cite{Alemdar2013}, topological groups with operations \cite{MucukandSahan}, and irresolute topological groups \cite{BasarAkiz2025}, and was also studied in the context of groupoids in \cite{Alemdar2012}. Through this approach, invariants such as connectedness and holes of a topological space are analyzed using the fundamental group. Moreover, the collection of fundamental groups in a space can be interpreted as the objects of a groupoid, where every morphism is an isomorphism. Our aim, however, is to adapt this method to the framework of soft topological groups by constructing the notion of the soft fundamental group.

In this study, the concept of soft connectedness, defined in terms of set decomposition, was introduced and its fundamental properties were investigated in soft topological spaces.  Subsequently, the soft usual topology compatible with the usual topology of $\mathbb{R}$ was defined. Based on this definition, by considering the soft subspace topology on the interval $[0,1]$, the notion of a soft path was established and the concept of soft path connectedness was introduced. This construction provides the theoretical foundation necessary for extending the concepts of homotopy and fundamental group from classical topology to the context of soft topology. Furthermore, the behavior of soft connectedness and soft path connectedness within soft topological groups was analyzed, and it was proven that these properties are preserved under soft continuous mappings.

Finally, the category of soft topological groups was constructed, its morphisms were identified, and the existence of terminal and product objects within this category was demonstrated. The results show that the category of soft topological groups forms a symmetric monoidal category.

\section{Preliminaries}

\begin{defn}\label{}
\em{Let \( X \) be the initial universe, \( \mathcal{P}(X) \) be the power set of \( X \), and \( \xi \) be  the set of parameters.
Let \( W : \xi \to \mathcal{P}(X) \) be a set-valued function.
Then, the set
\[
 W_\xi = \{ (e, W(e)) : e \in \xi,\ W(e) \in \mathcal{P}(X) \}
\]
is called a \textbf{soft set over $X$}~\cite{Mol}.

}
\end{defn}
\begin{defn}\label{}
\em{ Let \( W_\xi \) and \( W'_{\xi'} \) be two soft sets over  \( X \).
 \begin{enumerate}
  \item If \( W(e) = \emptyset \), for every \( e \in \xi \), then the soft set $W_\xi$ is called a \textbf{soft empty set}, and it is denoted by \( \emptyset_\xi \).
  \item If \( W(e) = X \), for every \( e \in \xi \), then the soft set $W_\xi$ is called an \textbf{absolute soft set}, and it is denoted by \( X_\xi \).
  \item If \( \xi \subseteq \xi' \) and \( W(e) \subseteq W'(e) \), for every \( e\in \xi \) , then \( W_{\xi} \) is called a \textbf{soft subset} of \( W'_{\xi'} \), and it is denoted by \( W_{\xi}  \widetilde{\sqsubseteq} W'_{\xi'}\).
   \item If  \(W_{\xi}  \widetilde{\sqsubseteq}W'_{\xi'} \) and  \(W'_{\xi'}  \widetilde{\sqsubseteq}W_{\xi}\), then \( W_{\xi} \) is said to be \textbf{soft equal} to \( W'_{\xi'} \) and  is denoted by
\(W_{\xi} = W'_{\xi'}\).
  \item Let \( \xi'' = \xi \cap \xi' \).
If, for each \( e \in \xi'' \), \(
W''(e) = W(e) \cap W'(e),
\) then the soft set \( W''_{\xi''} \) is called the \textbf{soft intersection} of \( W_\xi \) and \( W'_{\xi'} \),
and is denoted by \( W_\xi  \widetilde{\sqcap} W'_{\xi'} \).
  \item Let \( \xi''' = \xi \cup \xi' \).
If, for each \( e \in \xi''' \):
\[
W'''(e) =
\begin{cases}
W(e), & \text{if } e \in \xi \setminus \xi' \\
W'(e), & \text{if } e \in \xi \setminus \xi'' \\
W(e) \cup W(e), & \text{if } e \in \xi \cap \xi'
\end{cases}
\]then \( W'''_{\xi'''} \) is called the \textbf{soft union} of the soft sets \( W_\xi \) and \( W'_{\xi'} \),
and is denoted by \( W_\xi  \widetilde{\sqcup} W'_{\xi'} \).
\item  For each \( e \in \xi \), define \(
W^c(e) = X \setminus W(e)
\). Then \( {W_\xi^c} \) is called the \textbf{soft complement} of the soft set \( W_\xi \)~\cite{Maji}.
\end{enumerate}
}
\end{defn}

\begin{defn}
\em{
Let \( W_\xi \) be a soft set over  \( X \) and let \( x \in X \).
If \( x \in W(e) \) for every \( e \in \xi \), then \( x \) is called a \textbf{soft element} of the soft set \( W_\xi \), and is denoted by
\(
x \widetilde{\in}W_\xi
\).
Otherwise, if there exists at least one \( e \in \xi \) such that \( x \notin W(e) \), then \( x \) is not a soft element of \( W_\xi \), and is denoted by
\(
x \widetilde{\notin}W_\xi
\)
\cite{Naz}.
}
\end{defn}


As a convention, throughout this paper, we denote by $P^S(X)$ the collection of all soft sets over the initial universe $X$.

\begin{defn}\em{ Let \( W_\xi  \in P^S(X) \) and \( U_{\mathscr{D}} \in P^S(Y) \), and let  \( \varrho: X \rightarrow Y \) and \( \varphi: \xi \rightarrow \xi' \) be two functions.
\begin{itemize}
 \item[(i)]
\( (\varphi, \varrho): W_\xi \rightarrow U_{\mathscr{D}} \) is said to be a \textbf{soft mapping} from \( W_\xi \) to \( U_{\mathscr{D}} \)
if and only if
\[
\varrho(W_\xi(e)) = U_{\mathscr{D}}(\varphi(e)),
\]
for all $ e \in \xi.$
\item[(ii)]
The \textbf{soft image} of the soft set \( W_\xi \in P^S(X) \) under the mapping \( \varrho\) with respect to the mapping \( \varphi \)
is the soft set \( \varrho(W_\xi)_{\varphi(\xi)} \in  P^S(Y) \), defined by
\[
\varrho(W_\xi)_{\varphi(\xi)}(d) = \bigcup_{\varphi(e) = d} \varrho(W_\xi(e)),\]
for all \(d \in \varphi(\xi).\)
\item[(iii)]
The \textbf{inverse soft image} of the soft set \(U_{\mathscr{D}} \in P^S(Y) \) under the mapping \( \varrho \) with respect to the mapping \( \varphi \)
is the soft set \( \varrho^{-1}(U_{\mathscr{D}})_\xi \in P^S(X) \), defined by
\[
\varrho^{-1}(U_{\mathscr{D}})_\xi(e) = \varrho^{-1}(U_{\mathscr{D}}(\varphi(e))),
\]
for all $e \in \xi$ \cite{Kha}.

\end{itemize}
}
\end{defn}

This definition was arranged by Tahat at al.~\cite{Tahat}  as in the remark below:
\begin{rem}\em{
 Let \( W_\xi  \in P^S(X) \) and \( U_{\mathscr{D}} \in P^S(Y) \), and let  \( \varrho: X \rightarrow Y \) and \( \varphi: \xi \rightarrow \xi'\) be two functions.
 \begin{itemize}
 \item[(i)]
Let \( (\varphi \times \varrho) : \xi \times X \to {\mathscr{D}} \times Y \) be a soft mapping defined by
\(
(\varphi \times \varrho)(e, x) = (\varphi(e), \varrho(x))\), for all \( e \in \xi,\, x \in X.
\)
Then,
\[
(\varrho(W_\xi))_{\varphi(\xi)} = (\varphi \times \varrho)(W_\xi).
\]
\item[(ii)]
Let \( (\varphi \times \varrho) : \xi \times X \to {\mathscr{D}} \times Y \) be a soft mapping defined by
\((\varphi \times \varrho)(e, x) = (\varphi(e), \varrho(x))\),for all \(e \in \xi,\ x \in X.\)
Then, the inverse image of \( U_{\mathscr{D}} \) under \( \varrho \) with respect to \( \varphi \) satisfies
\[
(\varrho^{-1}(U_{\mathscr{D}}))_\xi = (\varphi \times \varrho)^{-1}(U_{\mathscr{D}}).
\]

\item[(iii)]
Let \( \xi= {\mathscr{D}} \) and \( \varphi = \mathbb{I}_\xi\) be the identity mapping on \( \xi \).
Then for all \( e \in \xi \), the image of the soft set \( W_\xi \) under \( \varrho \) with respect to \( \varphi \) satisfies $$(\varrho(W_\xi))_{\varphi(\xi)}(e) = \varrho(W_\xi(e)).$$

\end{itemize}
}
\end{rem}

\begin{defn}\label{}
\em{Let $X$  be the initial universe and $\xi$ be a set of parameters. A collection $\zeta$  of soft sets over  $X$ is called a \textbf{soft topology} on \( X \) if;
\begin{itemize}
  \item  \( {\emptyset}_\xi \) and  \( X_\xi \) belongs to \( \zeta\),
  \item The soft union of any family of soft sets in \( \zeta\) belongs to \( \zeta \),
  \item The soft intersection of any two soft sets in \( \zeta\) belongs to \( \zeta \).
\end{itemize}
The pair \( (X, \zeta)_\xi \) is called a \textbf{soft topological space}. The members of
 \( \zeta \) are called  \textbf{soft open sets}.
A soft set \( W \) over  \(X\) is called \textbf{soft closed } if soft complement \( W^c \) soft open
\cite{Naz}.
}
\end{defn}

\begin{ex}\em{
 Let \( X \) be an initial universal set and \( \xi \) be a parameter set. The collection \( \{{X}_\xi,\emptyset_\xi  \} \), consisting of the soft empty set and soft absolute set, is a soft topology on X, called \textbf{soft trivial topology}
  or \textbf{soft indiscrete topology} on  $X$ \cite{Naz}.
  }
\end{ex}

\begin{prop}
Let \( (X, \zeta)_\xi \) be a soft topological space.
Then, for each \( e \in \xi\)

\[
\zeta_e = \{ W(e) \mid W \in \zeta \}
\]
is a topology on \( X \) \cite{Naz}.
\end{prop}
\begin{defn}\em{
Let \( (X, \zeta)_\xi \) be a soft topological space, and let \( A\subseteq X\) be a nonempty subset.
Define a mapping \( W_A:\xi \to \mathcal{P}(X)\) by
\(
W_A(e) = W(e) \cap A
\)
for all $W \in \zeta$.
Then the collection

\[
\zeta_A = \{ W_A \mid W \in \zeta \}
\]
is called the \textbf{soft subtopology} on \( A \), and the pair \( (A, \zeta_A)_\xi \) is called a \textbf{soft subspace} of \( (X, \zeta)_\xi \) \cite{Naz}.
}
\end{defn}
\begin{defn}\em{
 Let  \( (X, \zeta)_\xi \)  be a soft topological space,  \(x\in X \), and $U$ be a soft set on $X$ such that \( x \widetilde{\in} U \).
If there exists a soft open set $W$ such that \( x \widetilde{\in} W \widetilde{\sqsubseteq} U\), then the soft set \( U \) is called a \textbf{soft neighborhood} of \( x \) \cite{Naz}.
}
\end{defn}

\begin{defn}\em{
Let \( \mathscr{B} \) be a collection of soft sets over \( X\) which is closed under finite soft intersection.
Then \[
\zeta = \{ {\emptyset}_\xi, {X}_\xi  \} \cup \left\{ \widetilde{\bigsqcup} B \,\middle|\, B \subseteq \mathscr{B} \right\}
\] is the soft topology induced by the basis \( \mathscr{B} \)
\cite{Aygun}.
}
\end{defn}
\begin{defn}\em{
Let \( \mathscr{S} \) be an arbitrary collection of soft sets over \( X \).
Then the collection
\[
\mathscr{B}=\left\{ W_1 \widetilde{\sqcap} W_2 \widetilde{\sqcap} \dots \widetilde{\sqcap} W_k \,\middle|\, W_1, W_2, \dots,W_k \in \mathscr{S} \right\}
\]
is a soft basis for a soft topology on \( X \) \cite{Aygun}.
}
\end{defn}
\begin{defn}\em{
Let \( (X, \zeta)_\xi \) and \( (X', \zeta')_{\xi'} \) be two soft topological spaces. A soft mapping \[ (\varphi, \varrho):(X, \zeta)_\xi  \to (X', \zeta')_{\xi'} \] is \textbf{soft continuous} at the point $x\in X$ if for every soft open neighborhood \( W' \) of \( \varrho(x) \),
there exists a soft  open neighborhood \( W \) of \( x\) such that
\(
\varrho(x) \widetilde{\in} (\varphi, \varrho)(W) \widetilde{\sqsubseteq} W'
\). If $(\varphi, \varrho)$ is soft continuous at every point of $X$, then $(\varphi, \varrho)$ is called a \textbf{soft continuous mapping} \cite{Hida}.
}
\end{defn}

If \( \xi = \xi' \) and \( \varphi = \mathbb{I}_\xi \) (the identity mapping on \( \xi \)),
then the definition of soft continuity is as follows:
\begin{defn}\label{cont}\em{
Let \( \varrho: X \to X' \) be a function.
Then \( (\mathbb{I}_\xi,\varrho) \) is called a \textbf{soft continuous mapping} from \( (X, \zeta)_\xi \) to \( (X', \zeta')_\xi \)
if the following condition holds \cite{Hida}:

\begin{itemize}
    \item For every \( x \in X \), and for every soft open neighborhood \( W' \) of \( \varrho(x) \),
    there exists a soft open  neighborhood \( W \) of \( x \) such that
    \(
    \varrho(W) \widetilde{\sqsubseteq} W'.
    \)
\end{itemize}
}
\end{defn}
\begin{prop}\label{inversopen}
    Let \( (X, \zeta)_\xi \) and \( (X', \zeta')_\xi \) be two soft topological spaces and let \( \varrho: X \to X' \) be a function. Then the soft mapping \( (\mathbb{I}_\xi,\varrho)\colon (X, \zeta)_\xi \to (X', \zeta')_\xi \) is soft continuous if for every soft open set \( W' \in \zeta' \), the inverse image \( (\mathbb{I}_\xi,\varrho)^{-1}(W') \in \zeta \) \cite{Hida}.

\end{prop}
 The converse of this proposition is not true in general. Hida \cite{Hida} gives the following example which shows that this is not always true.

 \begin{ex}\em{
Let $X = \{v\}$ and $\xi = \{e_1,e_2\}$, and consider the soft topologies, $\zeta=\{X_\xi,\emptyset_\xi\}$ and $\zeta'=\{X_\xi,\emptyset_\xi,\{(e_2,\{v\})\}\}$.
Then, the soft mapping $(\mathbb{I}_\xi,\mathbb{I}_X) : (X, \zeta)_\xi \to (X', \zeta')_\xi$ is soft continuous. Although $\{(e_2,\{v\})\}$ is soft open in $\zeta'$,  $(\mathbb{I}_\xi,\mathbb{I}_X)^{-1}(\{(e_2,\{v\})\})=\{(e_2,\{v\})\}$ is not soft open in $\zeta$. Therefore
in a soft continuous mapping, the inverse image of a soft open set is not always soft open.}
 \end{ex}

\begin{defn}\em{
Let \( (X, \zeta)_\xi \) and \( (X', \zeta')_{\xi'} \) be two soft topological spaces.
Then the collection
\[
\{ W \widetilde{\times} W' \mid W \in \zeta, W' \in \zeta' \}
\]
induces a soft topology \( \zeta_\times \) on the cartesian product set \( X\times X' \).
The soft topological space \( (X \times X',\zeta_\times)_ {\xi \times \xi'} \)
is called the \textbf{soft product space} of \( (X, \zeta)_\xi \) and \( (X', \zeta')_{\xi'} \)
\cite{Hida}.}
\end{defn}

\begin{defn}\em{
Let \( (G,\ast) \) be a group. If there exists a soft topology $\zeta$ on $G$ with the parameter set $\xi$ such that

\begin{itemize}
    \item
 For every \( (a, b) \in G\times G \), and  for each soft open neighborhood \( W \) of \( a \ast b \),
    there exist soft open neighborhoods \( U \) of \( a\) and \( V \) of \( b \) such that:
    \(
    U \ast V \widetilde{\sqsubseteq} W,
    \)

    \item The inverse map \( \iota : G \to G \), defined by \( \iota(a) = a^{-1} \), is soft continuous,
\end{itemize}
then the pair \( ( G ,\zeta )_\xi\) is called a \textbf{soft topological group}
\cite{Hida}.}
\end{defn}

\begin{prop}\label{topgroup}
 \( ( G ,\zeta )_\xi\) is a soft topological group if and only if for every \( a, b \in G \) and for every soft open set \( W \) containing \( a\ast b^{-1} \), there exist soft open neighborhoods \( U \) of \( a\) and \( V \) of \( b \) such that
 \( U \ast (V)^{-1} \widetilde{\sqsubseteq} W
\) \cite{Hida}.
\end{prop}
\section{ Soft Connectedness}
In view of the results obtained in this paper, we now introduce a new definition of soft connectedness, which is more compatible with the classical notion of connectedness in topology. In general topology, the definition of disconnectedness is given in terms of the set partition of a set by the open sets of the topology. In this study, in a similar manner, we relate the partition of a set to soft open sets and introduce the definition of soft disconnectedness as follows.

\begin{rem}\em{Let \( X\) be the initial universe and \( \xi \) be the set of parameters. If $A\subseteq X$, and the map $A:\xi \to \mathcal{P}(X)$ is defined as $A(e)=A$ for each $e\in \xi$,  then $A_\xi$ is a  soft set over \( X \). The soft set $A_\xi$, which is a soft subset of the absolute soft set $X_\xi$, is called an \textbf{absolute soft subset}.}
\end{rem}

\begin{defn}\em{
Let $(X, \zeta)_\xi $ be a soft topological space. If there exist two nonempty subsets $A$ and $B$ of $X$ such that $A \cup B = X$ and $A \cap B = \emptyset$, with both absolute soft subset $A_\xi$ and $B_\xi$ being soft open, then the soft topological space $(X, \zeta)_\xi $ is called  \textbf{soft disconnected}. Otherwise, $(X, \zeta)_\xi $ is called a \textbf{soft connected topological space}.}
\end{defn}

Let $(X, \zeta)_\xi $ be a soft topological space. If $A$ and $B$ are nonempty disjoint subsets of $X$ such that $X=A \cup B$, then it is obvious from the definition of absolute soft subset that \(
X_\xi = A_\xi \widetilde{\sqcup} B_\xi\) and \(A_\xi\widetilde{\sqcap} B_\xi=\emptyset_\xi\). However, even if $U$ and $V$ are soft open  in $\zeta$ such that $U\widetilde{\sqcup} V=X_\xi$, and $U\widetilde{\sqcap} V=\emptyset_\xi$, it does not necessarily follow that $U$ and $V$ constitute a set partition of $X$.
\begin{ex}\em{
 Let $X = \{\upsilon,\upsilon',\upsilon''\}$ and $\xi = \{e_1,e_2\}$, and consider the soft topology $$\zeta=\{X_\xi,\emptyset_\xi,\{(e_1,\{\upsilon,\upsilon'\}), (e_2,\{\upsilon''\})\}, \{(e_1,\{\upsilon''\}), (e_2,\{\upsilon,\upsilon'\})\}\}$$ on $X$.  Here, although $$\{(e_1,\{\upsilon,\upsilon'\}), (e_2,\{\upsilon''\})\}\widetilde{\sqcup} \{(e_1,\{\upsilon''\}), (e_2,\{\upsilon,\upsilon'\})\}\}=X_\xi$$  and $$\{(e_1,\{\upsilon,\upsilon'\}), (e_2,\{\upsilon''\})\}\widetilde{\sqcap} \{(e_1,\{\upsilon''\}), (e_2,\{\upsilon,\upsilon'\})\}\}=\emptyset_\xi,$$ no element of $X$ belongs to either $\{(e_1,\{\upsilon,\upsilon'\}), (e_2,\{\upsilon''\})\}$ or $\{(e_1,\{\upsilon''\}), (e_2,\{\upsilon,\upsilon'\})\}$. Hence, a partition of the set $X$ cannot be obtained by these soft open sets.}
\end{ex}
Let $(X, \zeta)_\xi $ be a soft topological space and $A\subseteq X$. If the soft subspace $(A, \zeta_A)_\xi$  is soft connected, then $A$ is said to be soft connected.

\begin{thm}
Let $(X, \zeta)_\xi $ be a soft topological space. Then the following are equivalent:
\begin{enumerate}
\item $X$ is soft connected.
\item The only soft absolute subsets are both soft open and soft closed are $\emptyset_\xi$ and $ X_\xi$.
\end{enumerate}
\end{thm}

\begin{proof}
(1) $\Rightarrow$ (2).
Let $X$ be  soft connected. Suppose that $A_\xi$ is a both soft open and soft closed absolute subset of $X_\xi$ with $\emptyset_\xi \neq A_\xi \neq X_\xi$. Since $A_\xi $ is soft closed, its soft complement $ A_\xi ^{c}$ is also soft open. Then, the pair $ A$ and $ A^{c}$ are two nonempty, disjoint  subsets of $X$ whose union is $X$. This contradicts the soft connectedness of $X$. Therefore, except for $\emptyset_\xi$ and $ X_\xi$, there is no other soft absolute subset that is both soft open and soft closed.

(2) $\Rightarrow$ (1).
Assume that the only soft absolute subsets that are both soft open and soft closed  in $\zeta$ are $\emptyset_\xi$ and $ X_\xi$. If $X$ were soft disconnected,  there would exist absolute soft open sets $A_\xi$ and $B_\xi$ such that $X_\xi = A_\xi \widetilde{\sqcup} B_\xi$.
In this case, $A_\xi$ is soft open, and its soft complement is $B_\xi$, which is also soft open. Hence, $A_\xi$ is both soft open and soft closed, contradicting the assumption. Therefore, $X$ is soft connected.

\end{proof}

\begin{thm}\label{softcontconnectted}
Let \( (X, \zeta)_\xi \) and \( (X', \zeta')_{\xi} \) be soft topological spaces, and let $(\mathbb{I}_\xi,\varrho)\colon (X, \zeta)_\xi \to  (X', \zeta')_\xi $ be a soft continuous mapping.
If $X$ is soft connected, then $\varrho(X)$ is soft connected.
\end{thm}
\begin{proof}
 Let $X$ be soft connected. And assume that $(\varrho(X),\zeta'_{\varrho(X)})_\xi$ is a soft disconnected space. Then  there exists a soft absolute subset $A_\xi$  that is both soft open and soft closed (${A}_\xi^c$ is soft open), with $\emptyset_\xi \neq A_\xi \neq {\varrho(X)}_\xi$.  For each $x \in X$, either $\varrho(x) \widetilde{\in} A_\xi$ or $\varrho(x) \widetilde{\in} A_\xi^c$.
 If $\varrho(x) \widetilde{\in} A_\xi$ and $\varrho^{-1}(A) = B$  then since $(\mathbb{I}_\xi,\varrho )$ is soft continuous, there exists a soft open neighborhood  $B_\xi$  of $x$ such that  $$\varrho (x) \widetilde{\in}(\mathbb{I}_\xi,\varrho)( B_\xi)\widetilde{\sqsubseteq}A_\xi.$$
 And similarly, if  $\varrho(x) \widetilde{\in} A_\xi^c$  and $\varrho^{-1}(A^c) = B^c$, then there exists a soft open neighborhood  $ B_\xi^c$  of $x$ such that  $$\varrho (x) \widetilde{\in}(\mathbb{I}_\xi,\varrho)(  B_\xi^c)\widetilde{\sqsubseteq} A_\xi^c.$$
 Therefore, each $x$ belongs to only one of $B_\xi$ or $ B_\xi^c$.
 So, we have $B_\xi\widetilde{\sqcup} B_\xi^c=X_\xi$ and $B_\xi\widetilde{\sqcap} B_\xi^c=\emptyset_\xi$, which is a contradiction. Hence, $\varrho(X)$ is soft connected.

\end{proof}

\section{Soft Path and Soft Path Connectedness }
Let $\xi$ be any parameter set. For every  $a,b \in \mathbb{R}$, let the set-valued functions \(M^a_\xi:\xi \to \mathcal{P}(\mathbb{R})\) be defined as \(M^a_\xi(e) = (a,\infty)\) and \(N^b_\xi:\xi \to \mathcal{P}(\mathbb{R})\) be defined as \(N^b_\xi(e) = (-\infty,b)\), for all $e\in \xi$.
Clearly $M^a_\xi$ and $N^b_\xi$ are soft sets over $\mathbb{R}$. Consider the class of soft sets
\[\mathscr{S}_\xi = \{M^a_\xi \mid a \in \mathbb{R}\} {\bigcup} \{N^b_\xi \mid b \in \mathbb{R}\} ,\]
which forms a soft subbase for a soft topology on $\mathbb{R}$.  Since each open interval $(a,b)$ in $\mathbb{R}$ can be written as $(-\infty,b)\cap(a,\infty)$,
from the soft subbase $\mathscr{S}_\xi $ we can obtain the soft base
\[
\mathscr{B}_\xi  = \{H^{(a,b)}_\xi \mid a,b \in \mathbb{R},\, a<b\},
\]
where $H^{(a,b)}_\xi:\xi \to \mathcal{P}(\mathbb{R})$ is defined by
\(
H^{(a,b)}_\xi(e) = (a,b)
\) for all $e\in \xi$.
From this soft base, we obtain the soft topology
\[\mathscr{U}_\xi = \{H^G_\xi \mid G \subseteq \mathbb{R} \text{ is open in the usual topology}\},\]
where $H^G_\xi:\xi \to \mathcal{P}(\mathbb{R})$ is given by $H^G_\xi(e)=G$.
The soft topology obtained in this way is called the soft usual topology, and the pair $(\mathbb{R},\mathscr{U}_\xi)_\xi=(\mathbb{R},\mathscr{U}_\xi)$ is called the soft usual topological space.

For each parameter $e \in \xi$, the associated topology $\mathscr{U}_{\xi}^e$ is the usual topology on $\mathbb{R}$.
\begin{ex}\em{
 Let $\xi=\{e_1,e_2,e_3\}$. Then the soft usual topology corresponding to $\xi$ is $\mathscr{U}_\xi=\{H^G_\xi \mid G \subseteq \mathbb{R}, \quad G \text{ is open in the usual topology}\}$. If $G=(-3,2)$, then
 $H^G_\xi=H^{(-3,2)}_\xi=\{(e_1,(-3,2)),(e_2,(-3,2)),(e_3,(-3,2))\}$. If $G=\mathbb{R}$, then $H^\mathbb{R}_\xi=\{(e_1,\mathbb{R}),(e_2,\mathbb{R}),(e_3,\mathbb{R})\}$.}
\end{ex}
From the definition of the soft usual topology, one can easily see that  a soft open set in soft usual topology is also an absolute open subset.

\begin{defn}\em{
Let $(X, \zeta)_ \xi$ be a soft topological space and let $(I,(\mathscr{U}_{\xi})_{I})$ be the soft subspace of  $(\mathbb{R},\mathscr{U}_\xi)$ on $I = [0,1] \subseteq \mathbb{R}$. Let
$\gamma \colon I \to X$
be a function. We say that the soft mapping $$\Gamma=(\mathbb{I}_\xi,\gamma) :(I,(\mathscr{U}_{\xi})_{I})\to (X, \zeta)_ \xi$$  is a \textbf{soft path} from $\gamma(0)=x$ and $\gamma(1)=y$, if $\Gamma$ is soft continuous.}
\end{defn}
\begin{rem}\em{
Let us note that, in the definition of the soft path, the parameter sets of  the soft usual topology on $\mathbb{R}$ and the soft subtopology obtained on $I$ are the same as the parameter set of the soft space $(X, \zeta)_\xi$.}
\end{rem}

\begin{ex}\label{exam4.3}\em{
Let $(\mathbb{R},\mathscr{U}_\xi)$ be the soft usual topology with parameter set  $\xi = \{e_1, e_2\}$. Let $\gamma \colon I \to \mathbb{R}$ be the function defined by $\gamma(t)=2t $. Then \[
\Gamma =(\mathbb{I}_\xi,\gamma) \colon (I,(\mathscr{U}_{\xi})_{I})\to(\mathbb{R},\mathscr{U}_\xi)\] is a soft path from $\gamma(0) = 0 $ to $\gamma(1) = 2 $. To see that, it is sufficient to show that $\Gamma$ is soft continuous. Let $t\in I$ be any point. Let us consider the open interval $G_\epsilon=(\gamma(t)-\epsilon,\gamma(t)+\epsilon)=(2t-\epsilon,2t+\epsilon)$. Then, for each $\epsilon>0$, $H^{G_\epsilon}_\xi$ is a soft open set which is a soft open neighborhood of $\gamma(t)$. Now, let $\delta= \frac{\epsilon}{2}$. By the definition of soft subtopology, for every $e\in \xi$, $$H^{(t-\frac{\epsilon}{2},t+\frac{\epsilon}{2})}_\xi(e)\cap I=(t-\frac{\epsilon}{2},t+\frac{\epsilon}{2})\cap I =H^{(t-\frac{\epsilon}{2},t+\frac{\epsilon}{2})\cap I}_\xi(e),$$ and  $$(H^{G_\epsilon}_\xi)_I=H^{(t-\frac{\epsilon}{2},t+\frac{\epsilon}{2})\cap I}_\xi\in (\mathcal{U}_{\xi})_{I}$$  is a soft open neighborhood of $t$. Since $\gamma((t-\frac{\epsilon}{2},t+\frac{\epsilon}{2}))=(\gamma (t-\frac{\epsilon}{2}),\gamma (t+\frac{\epsilon}{2}))=(2t-2\frac{\epsilon}{2},2t+2\frac{\epsilon}{2})=(2t-\epsilon,2t+\epsilon)$, then     $\Gamma (H^{(t-\frac{\epsilon}{2},t+\frac{\epsilon}{2})\cap I}_\xi)\widetilde{\sqsubseteq}(H^{G_\epsilon}_\xi)_I$. Hence, $\Gamma$ is soft continuous.


}
\end{ex}
\begin{ex}\em{
 Let $X = \{v\}$ and $\xi = \{e_1,e_2\}$, and consider the topology $\zeta=\{X_\xi,\emptyset_\xi,\{(e_1,\{v\})\}\}$.
For the function $\gamma:I \to X$ defined by $\gamma(t)=v$ for all $t\in            I $, $\Gamma =(\mathbb{I}_\xi ,\gamma) \colon (I,(\mathscr{U}_{\xi})_{I})\to  (X,\zeta)_\xi$ is a soft path.
Although $U=\{e_1,\{v\}\} \in \zeta$, $\Gamma^{-1}(U)=\{(e_1,I),(e_2,\emptyset)\}$ is not soft open in $(\mathcal{U}_{\xi})_{I}$. Therefore
in a soft path the inverse image of a soft open set is not always soft open.}
\end{ex}
\begin{ex}\em{
Let $X = \{1, 2, 3\}$, and let $\xi = \{e_1, e_2\}$ be a set of parameters. Define a soft topology $\zeta$ over $X$ consisting of the following soft sets:
\[
\zeta = \left\{ \emptyset_\xi, X_\xi, U_1, U_2 \right\},
\]
where the soft  sets $U_1$ and $U_2$ are given as follows:
\begin{align*}
U_1(e_1) &= \{1\},       & U_1(e_2) &= \{2\}, \\
U_2(e_1) &= \{1, 3\},    & U_2(e_2) &= \{2\}.
\end{align*}
Now define the soft maps:
\[
\Gamma_{13} =(\mathbb{I}_\xi , \gamma_{13}) : (I,(\mathscr{U}_{\xi})_{I})\to  (X,\zeta)_\xi,
\] and \[
\Gamma_{22} =(\mathbb{I}_\xi , \gamma_{22}) : (I,(\mathscr{U}_{\xi})_{I})\to  (X,\zeta)_\xi,
\]
with the functions \(\gamma_{13},\gamma_{22}:I \to X \),
\[
\gamma_{13}(t) =
\begin{cases}
1, & 0 \leq t < \frac{1}{2}, \\
3, & \frac{1}{2} \leq t \leq 1,
\end{cases}
\qquad \gamma_{22}(t) = 2 \quad \text{for all } t \in I.
\]
Let us show that $\Gamma_{13}$ and $\Gamma_{22}$ are soft continuous by using Definition \ref{cont}.

If $t\in [0,\frac{1}{2})$, then $\gamma_{13}(t)=1$, and in $\zeta$ the only soft open set containing $1$ is $X_\xi$.  $H^{[0,\frac{1}{2})}_\xi$ is an soft open set in $(\mathscr{U}_{\xi})_{I}$ that consists $t$ and $1 \widetilde{\in}\Gamma_{13}(H^{[0,\frac{1}{2})}_\xi)\widetilde{\sqsubseteq} X_\xi$. Therefore, for every element of $[0,\frac{1}{2})$, $\Gamma_{13}$ is soft continuous. Similarly, it is shown that $\Gamma_{13}$ is soft continuous for every $t \in [\frac{1}{2},1]$. Hence $\Gamma_{13}$ is a soft path.

If $t\in [0,1]$, then $\gamma_{22}(t)=2$, and in $\zeta$ the only open set containing $2$ is $X_\xi$.  $H^{[0,1]}_\xi$ is an soft open set in $(\mathscr{U}_{\xi})_{I}$ that consists $t$ and $2 \widetilde{\in}\Gamma_{22}(H^{[0,1]}_\xi)\widetilde{\sqsubseteq} X_\xi$. Therefore, for every element of $[0,1]$, $\Gamma_{22}$ is soft continuous. Therefore $\Gamma_{22}$ is a soft path.

Now,let us consider the soft open set $U_1$. Then we obtain
\[
\Gamma_{13}^{-1}(U_1) = \{(e_1, \gamma_{13}^{-1}\{1\}), (e_2, \gamma_{13}^{-1}\{2\})\} = \{(e_1, [0,\tfrac{1}{2})), (e_2, \emptyset)\},
\]
and

\[
\Gamma_{22}^{-1}(U_1) = \{(e_1, \gamma_{22}^{-1}\{1\}), (e_2, \gamma_{22}^{-1}\{2\})\} = \{(e_1, \emptyset), (e_2, I)\}.
\]
So, the inverse images of the soft open set $U_1$  are not soft open sets.}
\end{ex}

\begin{defn}\em{
Let $(X,\zeta)_\xi$ be a soft topological space. If for every $x,y \in X$  there exists a soft path \[
 \Gamma =(\mathbb{I}_\xi ,\gamma) \colon (I,(\mathscr{U}_{\xi})_{I})\to  (X,\zeta)_\xi,
\] from $x$ to $y$, then $(X,\zeta)_\xi$ is called a \textbf{soft path connected space}.}
\end{defn}

\begin{ex}\em{According to the soft usual topology defined above, the soft usual topological space
$(\mathbb{R}, \mathscr{U}_\xi)$ is a soft path connected space. Indeed, for every
$a, b \in \mathbb{R}$ there exists a function
\[
\gamma \colon I \to \mathbb{R}, \quad t \mapsto (1-t)a + tb,
\]
from $a$ to $b$. Similar to Example \ref{exam4.3}, it can be shown that \[
\Gamma =(\mathbb{I}_\xi ,\gamma) \colon (I,(\mathscr{U}_{\xi})_{I})\to  (\mathbb{R}, \mathscr{U}_\xi)
\] is a soft path.
Hence, $(\mathbb{R}, \mathscr{U}_\xi)$ is soft path connected.}
\end{ex}

 \begin{prop}\label{restricted}
  Let $(X,\zeta)_\xi$ and $(X',\zeta')_{\xi'}$ be soft topological spaces and  $(\varphi, \varrho): (X,\zeta)_\xi \to (X',\zeta')_{\xi'}$ be a soft continuous mapping. If $A$ is a non-empty subset of $X$ endowed with the soft subspace topology $\zeta_A$ , then  the soft mapping $$(\varphi, \varrho|_A): (A, \zeta_A)_\xi \to (X',\zeta')_{\xi'}$$
is also soft continuous, where $\varrho|_A$ denotes the restriction of $\varrho$ to $A$.
 \end{prop}
\begin{proof}
 Let $a\in A \subset X$ and  $\varrho(a)\widetilde{\in}U'$ be an arbitrary soft open set in $\zeta'$. Since $(\varphi, \varrho)$ is soft continuous, there is a soft open  neighborhood \( U \) of \( a \) such that
 \((\varphi, \varrho)(U) \widetilde{\sqsubseteq} U'\). By the definition of the soft subspace topology, $(U _ A) \in \zeta_{A}$, and hence $(\varphi, \varrho)(U_ A) \widetilde{\sqsubseteq} U'$. So $(\varphi, \varrho|_A)$ is soft continuous.
\end{proof}

\begin{thm}\label{preservepathcon}
Let $(X,\zeta)_\xi$ and $(X',\zeta')_{\xi}$ be soft topological spaces, and $A\subseteq X$.
If $(\mathbb{I}_\xi,\varrho):(X,\zeta)_\xi\to (X',\zeta')_{\xi}$ is a soft continuous function and the soft  subspace $(A,\zeta_A)_\xi$ is soft path connected, then the soft subspace $(\varrho(A),\zeta'_{\varrho(A)})_{\xi}$ is also soft path connected.
\end{thm}

\begin{proof}
Let $b_1, b_2 \in \varrho(A)$. Then there exist $a_1, a_2 \in A$ such that $\varrho(a_1)=b_1$ and $\varrho(a_2)=b_2$.
Since $(A,\zeta_A)_\xi$ is soft path connected, there exists a soft path
\[
\Gamma =(\mathbb{I}_\xi ,\gamma) \colon (I,(\mathscr{U}_{\xi})_{I})\to  (A,\zeta_A)_\xi,
\]
where
\[
\gamma: I \to A, \quad \gamma(0)=a_1, \ \gamma(1)=a_2.
\]

Now consider the composition
\[
(I,(\mathscr{U}_{\xi})_{I}) \;\xrightarrow{\;\Gamma\;}\; (A,\zeta_A)_\xi \;\xrightarrow{\;(\mathbb{I}_\xi, \varrho|_A)\;}\;(\varrho(A),\zeta'_{\varrho(A)})_{\xi}.
\]
This yields the soft path
\[
(\mathbb{I}_\xi, \varrho|_A)\circ\Gamma = (\mathbb{I}_\xi , (\varrho|_A \circ \gamma) ) \colon (I,(\mathscr{U}_{\xi})_{I}) \to (\varrho(A),\zeta'_{\varrho(A)})_{\xi},
\]
where \[\varrho|_A\circ \gamma: I \to \varrho(A)\] is a function with $(\varrho|_A\circ\gamma)(0)=\varrho(a_1)=b_1$ and $(\varrho|_A\circ\gamma)(1)=\varrho(a_2)=b_2$.

Since both $\Gamma$ and $(\mathbb{I}_\xi, \varrho|_A)$ are soft continuous, their composition is also soft continuous.
Therefore, there exists a soft path in $\varrho(A)$ from $b_1$ to $b_2$.

Hence $(\varrho(A),\zeta'_{\varrho(A)})_{\xi}$ is soft path connected.
\end{proof}
\begin{prop}
 Let $A\subset \mathbb{R}$ be a proper subset, and let $\xi$ be any parameter set. If $A$ is soft connected in the soft usual topological space $(\mathbb{R},\mathscr{U}_\xi)$, then $A$ is an interval.
\end{prop}

\begin{proof}
 Suppose that $A$ is not an interval. Now, consider the soft open sets $$U_\xi=H^{(-\infty,x)\cap A}_\xi \quad  \text{and} \quad V_\xi=H^{(x,\infty)\cap A}_\xi$$ in soft subspace $ (\mathscr{U}_{\xi})_{A}$ where $x\in \mathbb{R}$ and $x\notin A$.
Then, we obtain \(U_\xi\widetilde{\sqcup} V_\xi=A_\xi\), \(U_\xi\widetilde{\sqcap} V_\xi=\emptyset_\xi\),  which implies that $A$ is soft disconnected. This contradicts the assumption that \(A\) is not an interval, so \(A\) is an interval.
\end{proof}

\begin{thm}
Let $(X,\zeta)_\xi$ be a soft connected topological space, and let $$(\mathbb{I}_\xi,\varrho)\colon (X,\zeta)_\xi\to (\mathbb{R}, \mathscr{U}_\xi) $$ be a soft continuous mapping. Let $a, b \in \varrho(\mathbb{R})$. Then the function $\varrho$ takes all values between $a$ and $b$.
\end{thm}
 \begin{proof}
Since $X$ is soft connected, $\varrho(X)$ is also soft connected. If the function $\varrho$ is onto, the proof is complete. If $\varrho$ is not onto, then $\varrho(X)$ is an interval. Hence, for each $c\in [a , b]\subseteq \varrho(X)$, there exists an $x\in X$ such that $\varrho(x) = c$.

 \end{proof}

 \begin{prop}
 The soft usual topological space  $(\mathbb{R},\mathscr{U}_\xi)$  is soft connected.
 \end{prop}
 \begin{proof}
 Suppose that $\mathbb{R}$ is soft disconnected. Then there exist soft open sets $H^G_\xi$ and $H^{G'}_\xi$ such that $H^G_\xi\widetilde{\sqcup}H^{G'}_\xi= H^{G\cup G'}_\xi=\mathbb{R}_\xi$ and $H^G_\xi\widetilde{\sqcap}H^{G'}_\xi =H^{G\cap G'}_\xi=\emptyset_\xi$. From this, for the open sets $G$ and $G'$ of the usual topology on $\mathbb{R}$, $G\cup G'=\mathbb{R}$ and $G\cap G'=\emptyset$, which contradicts the connectedness of $\mathbb{R}$ with respect to the usual topology. Hence, $\mathbb{R}$ is soft connected.

 \end{proof}

 \begin{prop}
 $(I,(\mathscr{U}_\xi)_I)$ is a soft connected space.
\end{prop}

\begin{proof}
Suppose that $(I,(\mathscr{U}_\xi)_I)$  is soft disconnected. Then there exist  soft open sets $H^{G\cap I}_\xi$ and $H^{G'\cap I}_\xi$ in $(\mathscr{U}_\xi)_I$ such that
\[
H^{G\cap I}_\xi \widetilde{\sqcup}\ H^{G'\cap I}_\xi=I_\xi\quad \text{and} \quad H^{G\cap I}_\xi\widetilde{\sqcap} H^{G'\cap I}_\xi=\emptyset_\xi.
\]
$G\cap I$ and $G'\cap I$ are disjoint open sets in the usual topology of $I$ and $(G\cap I)\cup (G'\cap I) = I$.
This implies that $I$ is disconnected with respect to the usual topology, which is a contradiction. Therefore, $(I,(\mathscr{U}_\xi)_I)$ is a soft connected space.

\end{proof}

\begin{thm}
If $(X,\zeta)_\xi$ is a soft path connected space, then $(X,\zeta)_\xi$ is a soft connected space.
\end{thm}

\begin{proof}
Assume that $X$ is soft disconnected, then there exist absolute soft open subsets such that  \(
 A_\xi \widetilde{\sqcup} B_\xi=X_\xi \) and  \(A_\xi\widetilde{\sqcap} B_\xi=\emptyset_\xi\).
 Choose $x\in A$ and $y\in B$  and let \[
\Gamma=(\mathbb{I}_\xi, \gamma) :(I,(\mathscr{U}_\xi)_I)\to (X,\zeta)_\xi
\] be a soft path from $x$ to $y$. Since $(\mathbb{I}_\xi, \gamma)$ is soft continuous, $\gamma(I)$ is soft connected; however, $$(A_\xi\widetilde{\sqcap}\gamma(I)_\xi)\widetilde{\sqcup}(B_\xi\widetilde{\sqcap}\gamma(I)_\xi)=\gamma(I)_\xi \quad \text{and} \quad (A_\xi\widetilde{\sqcap}\gamma(I)_\xi)\widetilde{\sqcap}(B_\xi\widetilde{\sqcap}\gamma(I)_\xi)=\emptyset_\xi $$ implies that $\gamma(I)$ is soft disconnected, which is a contradiction. Thus,
$(X,\zeta)_\xi$ is a  connected space.
\end{proof}

\section{ Soft Topological Groups}
Now we will investigate the concepts of soft connectedness and soft path connectedness in the context of soft topological groups.

\begin{defn}\em{
Let $ (G,\zeta)_\xi$  be a soft topological group. Then $ (G,\zeta)_\xi$ is called a
soft connected topological group if the underlying soft topological space
$ (G,\zeta)_\xi$ is soft connected.}
\end{defn}

\begin{thm}\label{grpoperationcont}Let $(G,\zeta)_\xi$ be a soft topological group. Then the soft group operation
\[\ast\colon G\times G\to G, (a,b)\mapsto a\ast b\]
is soft continuous.
\end{thm}
\begin{proof} Suppose that  $(G,\zeta)_\xi$ is a soft topological group.  Then for every  $(a,b) \in G \times G $  and for every soft neighborhood  $W$ of $a \ast b$, there are soft open neighborhoods  $U$ of $a$  and  $V$  of  $b$ such that  $U \ast V  \widetilde{\sqsubset} W$. In the soft product space $(G\times G,\zeta_\times)_{\xi\times \xi}$,  $U\widetilde{\times} V$
 is a soft open neighborhood of $(a,b)$ and $\ast(U\widetilde{\times} V)=U \ast V  \widetilde{\sqsubseteq} W$. It shows that $\ast$ is soft continuous.
 \end{proof}
\begin{prop}\label{softsubgrp}
 Let $(G,\zeta)_\xi$ be a soft topological group.  If $H$ is a subgroup of $G$, then \((  H,\zeta_H)_\xi \) is a soft topological group.
\end{prop}
\begin{proof}
Let \(H \)  be a subgroup of the soft topological group $(G,\zeta)_\xi$ equipped with the soft subspace topology.
 Since $H$ is a subgroup, for every $x,y\in H$, $xy^{-1}\in H\subseteq G$. And since $G$ is a soft topological group, for every \( x, y \in G \) and for every soft open set \( W \) containing \( x y^{-1} \), there are soft open sets \( U \) and \( V \) such that
\(
\quad x \widetilde{\in}  U, \quad y \widetilde{\in} V\) and \( U V^{-1} \widetilde{\sqsubseteq} W
\). Then, by the definition of the soft subspace topology, for every \( x, y \in H \) and for every soft open set \( W_H \) containing \( x y^{-1} \), there are soft open sets \( U_H  \) and \( V_H  \) such that
\(
\quad x \widetilde{\in}  U_H  , \quad y \widetilde{\in} V_H \) and \( (U_ H ) {(V_ H)} ^{-1} \widetilde{\sqsubseteq} W_ H
\). Hence \((  H,\zeta_H)_\xi \) is a soft topological group.
\end{proof}
\begin{prop}\label{openclosed}
 Let $(G,\zeta)_\xi$ be a soft topological group and $H$ be a subgroup of $G$ such that $H_\xi$ is soft open. Then $H_\xi$ is soft closed \cite{Hida}.
\end{prop}
\begin{prop}
 A soft connected soft topological group
$(G,\zeta)_\xi$ has no soft subgroup
$H\subseteq G$
such that $H_\xi$ is soft open.
\end{prop}
\begin{proof}
 Let \( (G,\zeta)_\xi \) be a soft connected soft topological group. Assume that \( H \) is a subgroup of \( G \) such that \( H_\xi \) is soft open. Then, by Theorem \ref{openclosed}, \( H_\xi \) is soft closed. This contradicts the soft connectedness of \( G \). Hence, \( G \) has no subgroup \( H \) such that \( H_\xi \) is soft open.
\end{proof}

\begin{prop}\label{productcont}    Let $(X,\zeta)_\xi$ be a soft topological space and $(G,\zeta')_\xi$ be a  soft topological  group. If
\((\mathbb{I}_\xi , \varrho)\colon (X,\zeta)_\xi\to (G,\zeta')_\xi \) and
\((\mathbb{I}_\xi , \varrho')\colon (X,\zeta)_\xi\to (G,\zeta')_\xi\) are soft continuous mappings, then
\[(\mathbb{I}_\xi , \varrho\ast \varrho')\colon (X,\zeta)_\xi\to (G,\zeta')_\xi \] is also a soft continuous mapping.
\end{prop}

\begin{proof}
Let \((\mathbb{I}_\xi , \varrho)\colon (X,\zeta)_\xi\to (G,\zeta')_\xi \) and
\((\mathbb{I}_\xi , \varrho')\colon (X,\zeta)_\xi\to (G,\zeta')_\xi\) be  soft continuous
mappings. Now consider the soft mapping $$(\Delta_\xi, (\varrho,\varrho') ):
(X,\zeta)_\xi\to (G\times G,\zeta'_\times)_{\xi\times \xi},$$  for the functions
$\Delta_\xi\colon \xi \to \xi\times \xi$ defined by $\Delta_\xi(e)=(e,e)$ and
$(\varrho,\varrho'):X \to G\times G$ defined by $(\varrho,\varrho')(x)=(\varrho(x),\varrho'(x))$,where
$\zeta'_\times$ is soft product topology induced by $\zeta'$.

Since
$(\mathbb{I}_\xi , \varrho)$ is soft continuous, then  for every \( x \in X\) and
every soft open neighborhood \( U' \) of \( \varrho(x) \),
there is a soft  open neighborhood \( U \) of \( x \) such that
\(
\varrho(x) \widetilde{\in}\varrho(U) \widetilde{\sqsubseteq} U'
\). Similarly, since $(\mathbb{I}_\xi , \varrho')$ is soft continuous, for every soft open set $V'$ containing the point $\varrho(x)$,
there is a soft open set $V$ containing $x$ such that $\varrho'(x) \widetilde{\in}\varrho(V) \widetilde{\sqsubseteq} V'$.  So   for every \( x \in X\) and every soft open neighborhood \( U'\widetilde{\times}V' \) of \( (\varrho(x),\varrho'(x)) \),
there is a soft  open neighborhood \( U\widetilde{\sqcap}V \) of \( x \) such that
\(
(\varrho,\varrho')(x) \widetilde{\in} (\varrho,\varrho')(U\widetilde{\sqcap}V) \widetilde{\sqsubseteq} U'\widetilde{\times}V'
\). Hence $(\Delta_\xi, (\varrho,\varrho') )$ is a soft continuous mapping.

Moreover, by Proposition \ref{grpoperationcont}, the group operation
\[
\ast\colon G\times G\to G, \quad (a,b)\mapsto a\ast b
\]
is soft continuous. Therefore, the composition

\[
\begin{aligned}
(X,\zeta)_\xi
&\;\xrightarrow{\,(\Delta_\xi, (\varrho,\varrho') )\,}\;
(G\times G,\zeta'_{\times})_{\xi\times \xi}
\;\xrightarrow{\;\ast\;}\;
(G,\zeta')_\xi   \\[6pt]
x\;&\;\longmapsto\;(\varrho(x),\varrho'(x))\;\longmapsto\;\varrho(x)\ast \varrho'(x).
\end{aligned}
\]
is soft continuous. Hence, we conclude that
\[
(\ast)\circ((\Delta_\xi , (\varrho , \varrho'))=(\mathbb{I}_\xi , \varrho\ast \varrho')
\]
is soft continuous as well.

\end{proof}
\begin{ex}\em{
 Let $(G,\ast)$ be a group and let \( (G,\zeta)_\xi  \)  be a soft topological group. If $\Gamma:(I,(\mathscr{U}_\xi)_I) \to (G,\zeta)_\xi $ and $\Gamma':(I,(\mathscr{U}_\xi)_I) \to (G,\zeta)_\xi $ are  soft paths, then $\Gamma\ast\Gamma':(I,(\mathcal{U}_{\xi})_{I},\xi)\to (G,\zeta)_\xi  $ is also a soft path.}
\end{ex}
\begin{prop}
Let \( (G,\zeta)_\xi  \) be a soft topological group, and let $H$ and $K$
be soft connected subsets of $G$. Then  $H\ast K\subseteq G$ is also soft connected.
\end{prop}

\begin{proof}
Assume, to the contrary, that $H\ast K$ is soft disconnected.
Then there exist two nonempty, disjoint absolute soft open subsets
$A_\xi$ and $B_\xi$ of \( (G,\zeta)_\xi  \) such that $H\ast K = A \cup B, \quad A \cap B = \emptyset$.
For each $h\in H$, by \cite[Propsition 4.5]{Hida}, consider the left translation
$\alpha_L(h):G\to G$ defined by $\alpha_L(h)(x)=h\ast x$.
Since $\alpha_L(h)$ is  soft continuous,
the image $h\ast K=\alpha_L(h)(K)$ is soft connected for every $h\in H$.

Since $A_\xi$ and $B_\xi$ form a soft separation of $H\ast K$,
they induce a soft separation on each subset $h\ast K$.
By the soft connectedness of $h\ast K$, one must have
$h\ast K\subseteq A$ or $h\ast K\subseteq B$.

Suppose that there exist $h_1,h_2\in H$ with $h_1\ast K\subseteq A$
and $h_2\ast K\subseteq B$.
Choose $k'\in K$ and note that since the right translation is soft continuous, $H\ast k'$ is also soft connected.
Then $h_1\ast k'\in A$ and $h_2\ast k'\in B$. Therefore,  $h_1\ast k'\widetilde{\in} A_\xi$ and $h_2\ast k'\widetilde{\in} B_\xi$,
contradicting the soft connectedness of $H\ast k'$.
Hence, all $h\ast K$ subsets must lie in the same soft open set,
say $A_\xi$, implying
\[
H\ast K=\bigcup_{h\in H}h\ast K\subseteq A.
\]
Thus, $B_\xi=\emptyset_\xi$, contradicting the assumption that both
$A$ and $B$ are nonempty.
Therefore, $H\ast K$ is soft connected.
\end{proof}

\begin{thm}\label{softpathproduct}
Let $(X,\zeta)_\xi$ and $(X',\zeta')_{\xi'}$ be soft path connected spaces.
Then the soft product space $(X\times X',\zeta_\times )_{\xi\times \xi'}$ is also soft path connected.
\end{thm}

\begin{proof}
Let $(x_1,y_1),(x_2,y_2)\in X\times X'$.
Since $(X,\zeta)_\xi$ and $(X',\zeta')_{\xi}$ are soft path connected, there exist soft paths
\[
\Gamma_X=(\mathbb{I}_\xi,\gamma_X):(I,(\mathscr{U}_{\xi})_{I})\to (X, \zeta)_ \xi,
\qquad
\Gamma_{X'}=(\mathbb{I}_\xi',\gamma_{X'})(I,(\mathscr{U}_{\xi})_{I})\to (X', \zeta')_ {\xi'},
\]
such that $\gamma_X(0)=x_1$, $\gamma_X(1)=x_2$ and $\gamma_{X'}(0)=y_1$, $\gamma_{X'}(1)=y_2$.

Define a soft mapping
\[
(\mathbb{I}_{\xi\times \xi'},\gamma):(I,(\mathscr{U}_{\xi\times \xi'})_{I})\to (X\times X', \zeta_\times)_ {\xi\times \xi'}
\]
by
\[
\Gamma=\mathbb{I}_\xi\times \mathbb{I}_{\xi'}=\mathbb{I}_{\xi\times \xi'}: \xi\times \xi'\to \xi\times \xi'\quad \text{and} \quad \gamma(t)=(\gamma_X(t),\,\gamma_{X'}(t)).
\]

Since both $(\mathbb{I}_\xi,\gamma_X)$ and $(\mathbb{I}_\xi',\gamma_{X'})$ are soft continuous,
the soft continuity of $(\mathbb{I}_{\xi\times \xi'},\gamma)$ can be established in a manner similar to Theorem \ref{productcont}.

Moreover, $\gamma(0)=(x_1,y_1)$ and $\gamma(1)=(x_2,y_2)$.
Hence $\Gamma$ is a soft path in $X\times X'$ connecting $(x_1,y_1)$ and $(x_2,y_2)$.

Therefore, $(X\times X',\zeta_\times )_{\xi\times \xi'}$ is soft path connected.
\end{proof}

\begin{thm}
Let \( (G,\zeta)_\xi  \) be a soft topological group. If $U$ and $V$ are soft path connected subsets of $G$, then $U\ast V$ is also soft path connected.

\end{thm}

\begin{proof}
Since $U$ and $V$ are soft path connected, by Theorem \ref{softpathproduct}, their soft product $U\widetilde{\times} V$ is also soft path connected. On the other hand, by Theorem \ref{grpoperationcont} and Theorem \ref{preservepathcon}, since \[\ast\colon G\times G\to G, (a,b)\mapsto a\ast b\] is soft continuous, it follows that the image of  $U\widetilde{\times} V$ under this operation,  $U\ast V$
is also soft path connected.
\end{proof}

\begin{thm}
Let $(G,\zeta)_\xi $ and $(G',\zeta')_{\xi} $ be soft topological groups, and let
\[
\varrho\colon G\to G'
\]
be a group homomorphism.
Then $(\mathbb{I}_\xi,\varrho)$ is soft continuous if and only if it is soft continuous at the identity element $e_G$.
\end{thm}

\begin{proof}
\emph{($\Rightarrow$)} If $(\mathbb{I}_\xi,\varrho)$ is soft continuous on $G$, then it is certainly soft continuous at $e_G$.

\smallskip
\emph{($\Leftarrow$)} Assume that $(\mathbb{I}_\xi,\varrho)$ is soft continuous at $e_G$.
We need to show that $(\mathbb{I}_\xi,\varrho)$ is soft continuous at an arbitrary $a\in G$. Let $W'$ be a soft open neighborhood of $\varrho(a)$ in $\zeta'$.
By \cite[Propsition 4.5]{Hida},
\[
\alpha_L(a)\colon G\to G,\quad x\mapsto a\ast x,
\qquad\text{and}\qquad
\beta_{L}(\varrho(a))\colon G'\to G',\quad x\mapsto x\ast\varrho(a) .
\]
are soft homeomorphisms, then the soft  set $\beta_{L}(\varrho(a)^{-1}) \bigl (W')=\varrho(a)^{-1}\ast W'$
is a soft open neighborhood of the identity element $e_{G'}$ of the group $G'$.
So, there exists a soft neighborhood $W$ of $e_G$ in $(G,\zeta)_\xi $ such that $(\mathbb{I}_\xi,\varrho)(W)\ \widetilde{\sqsubset}\ W'$.
Thus, $\alpha_L(a)(W)=a\ast W$  is a soft open  neighborhood of $a$.
For any $a\in G$ we have
\[
\varrho(a)\widetilde{\in}(\mathbb{I}_\xi,\varrho)\bigl(a\ast W)
\widetilde{\sqsubset} W'.
\]
This proves the soft continuity of $\varrho$ at $a$.
Since $a$ was arbitrary, $(\mathbb{I}_\xi,\varrho)$ is soft continuous on $G$.
\end{proof}

\section{Category of Soft Topological Groups}

In this section, we first define the morphisms of the category whose objects are soft topological groups,
as well as the composition operation for these morphisms.
\begin{defn}\em{
Let $( G,\zeta)_\xi  $  and $( G',\zeta')_{\xi'} $ be soft topological groups. If $(\varphi,\varrho): (G,\zeta)_\xi   \to (G',\zeta')_{\xi'}$  is a soft continuous mapping such that $\varphi: \xi \rightarrow \xi'$ is a function and $\varrho: G \rightarrow G'$ is a group homomorphism, then $(\varphi,\varrho)$  is called a \textbf{morphism of soft topological groups}.}
\end{defn}
\begin{ex}\em{

Let \( (G,\zeta)_\xi\) be a soft topological group.
Let \( \mathbb{I}_\xi : \xi \to \xi \) be the identity function on the parameter set,
and \( \mathbb{I}_G : G \to G \) be the identity group homomorphism. Then the soft mapping
\[
(\mathbb{I}_\xi, \mathbb{I}_G) : (G,\zeta)_\xi \rightarrow  (G,\zeta)_\xi
\]
is soft continuous.
This morphism of soft topological groups is called the \textbf{identity morphism of soft topological groups} and is denoted by \(
\mathbb{I}_{ (G,\zeta)_\xi }.
\)
}
\end{ex}

\begin{defn}\em{
Let $(\varphi_1,\varrho_1) :  (G_1,\zeta_1 )_{\xi_1}  \rightarrow  (G_2,\zeta_2 )_{\xi_2} $ and $ (\varphi_2,\varrho_2) : (G_2,\zeta_2 )_{\xi_2}\rightarrow (G_3,\zeta_3 )_{\xi_3} $ be morphism of soft topological groups. Then $(\varphi_1,\varrho_1)=(\varphi_2,\varrho_2)$ if and only if $\varphi_1=\varphi_2$ and $\varrho_1=\varrho_2$.}
\end{defn}

\begin{prop}
Let \( (G_1,\zeta_1 )_{\xi_1} \), \( (G_2,\zeta_2 )_{\xi_2} \), and \( (G_3,\zeta_3 )_{\xi_3} \) be soft topological groups. Suppose that
\[(\varphi_1,\varrho_1) :  (G_1,\zeta_1 )_{\xi_1}  \rightarrow  (G_2,\zeta_2 )_{\xi_2} \quad \text{and} \quad
(\varphi_2,\varrho_2) : (G_2,\zeta_2 )_{\xi_2}\rightarrow (G_3,\zeta_3 )_{\xi_3}
\]
are morphisms of soft topological groups. Then the composition
\[
(\varphi_1,\varrho_1) \circ (\varphi_2,\varrho_2) = (\varphi_2\circ \varphi_1,\, \varrho_2 \circ \varrho_1)
\]
is also a morphism of soft topological groups.
\end{prop}
\begin{proof}
 Since \( \varphi_2\circ \varphi_1 \) is a function and \( \varrho_2 \circ \varrho_1 \) is a group homomorphism,
it is sufficient to show that \((\varphi_2\circ \varphi_1, \varrho_2 \circ \varrho_1)\) is a soft continuous mapping. And since \( (\varphi_1, \varrho_1) \) is a soft continuous mapping, for every \( a \in G_1 \) and for every soft open  neighborhood \( U' \) of \( \varrho_1(a) \),
there is a soft open neighborhood \( U \) of \( a \) such that \(
\varrho_1(a) \widetilde{\in} (\varphi ,\varrho)(U) \widetilde{\sqsubseteq} U'\).

Let \( b = \varrho_1(a) \). Because \((\varphi_2, \varrho_2)\) is a soft continuous mapping, for every \( b \in G_2 \) and for every  soft open neighborhood \( U'' \)  of \( \varrho_2(b)=\varrho_2(\varrho_1(a)) \) there is a soft open neighborhood \(U'\) of \(b\) such that
\(\varrho_2(\varrho_1(x)) \widetilde{\in} (\varphi_2\circ \varphi_1,\, \varrho_2 \circ \varrho_1)(U') \widetilde{\sqsubseteq} U''.\)

Hence, the composition $(\varphi_1,\varrho_1) \circ (\varphi_2,\varrho_2)$ is also soft continuous.

\end{proof}
\begin{thm}
The soft topological groups and the morphisms between the soft topological groups together form a category.
\end{thm}

\begin{proof}
 The objects of the category are all soft topological groups, and the morphisms are all morphisms of soft topological groups  between objects.

The partial composition of the category is defined as the composition of morphisms of soft topological groups.

For every object $(G,\zeta)_\xi$, there exists an identity morphism
\[
\mathbb{I}_{ (G,\zeta)_\xi }=(\mathbb{I}_\xi, \mathbb{I}_G) : (G,\zeta)_\xi \longrightarrow (G,\zeta)_\xi
\]
as given in Example 3.2.
For any morphism
\[
(\varphi, \varrho) : (G,\zeta)_\xi  \to ( G',\zeta')_{\xi'} ,
\]
the composition satisfies
\[(\varphi, \varrho) \circ \mathbb{I}_{(G,\zeta)_\xi } = (\varphi, \varrho) \circ(\mathbb{I}_\xi, \mathbb{I}_G)  = (\varphi \circ \mathbb{I}_\xi, \varrho \circ \mathbb{I}_G) = (\varphi, \varrho).\]
Similarly
\[\mathbb{I}_{( G',\zeta')_{\xi'}} \circ (\varphi, \varrho) =(\varphi, \varrho).\]

 Let  \( (G_1,\zeta_1 )_{\xi_1} \), \( (G_2,\zeta_2 )_{\xi_2} \), \( (G_3,\zeta_3 )_{\xi_3} \), and \( (G_4,\zeta_4 )_{\xi_4}\) be soft topological groups. And let
\(
(\varphi_1, \varrho_1) : (G_1,\zeta_1 )_{\xi_1}  \rightarrow(G_2,\zeta_2 )_{\xi_2} \),
\((\varphi_2, \varrho_2) : (G_2,\zeta_2 )_{\xi_2} \rightarrow (G_3,\zeta_3 )_{\xi_3}\) and \((\varphi_3, v_3) :(G_3,\zeta_3 )_{\xi_3}\rightarrow (G_4,\zeta_4 )_{\xi_4}
\)
be  morphisms of soft topological groups.
Since the composition of functions  and the composition of group homomorphisms  are associative, we also have
\begin{align*}
 (\varphi_3, \varrho_3) \circ \left((\varphi_2, \varrho_2) \circ (\varphi_1, \varrho_1)\right)
&= (\varphi_3, \varrho_3) \circ (\varphi_2 \circ \varphi_1,\ \varrho_2 \circ \varrho_1) \\
&=(\varphi_3 \circ (\varphi_2 \circ \varphi_1),\ \varrho_3 \circ (\varrho_2 \circ \varrho_1))\\
&= ((\varphi_3 \circ \varphi_2) \circ \varphi_1,( \varrho_3 \circ \varrho_2) \circ \varrho_1)  \\
&= (\varphi_3 \circ \varphi_2,\ \varrho_3 \circ \varrho_2) \circ (\varphi_1, \varrho_1)\\
&=((\varphi_3, \varrho_3)\circ(\varphi_2, \varrho_2))\circ (\varphi_1, \varrho_1).
\end{align*}
Thus, the composition of the category is associative.

\end{proof}
The category we have obtained here is called the category of soft topological groups and is denoted by $\mathbf{STGrp}$.

\begin{prop}
Let $(G,\zeta)_\xi  $ and $( G',\zeta')_{\xi'}$ be soft topological groups.
Then for the morphism of soft topological groups $(\varphi, \varrho) : (G,\zeta)_\xi  \to ( G',\zeta')_{\xi'}$:

\begin{itemize}
  \item[(i)] If $\varphi$ is an  injective function and $\varrho$ is an injective group homomorphism, then $(\varphi, \varrho)$ is monomorphism.
  \item[(ii)] If $\varphi$ is a surjective function and $\varrho$ is a surjective group homomorphism, then $(\varphi, \varrho)$ is an epimorphism.

\end{itemize}
\end{prop}
\begin{proof}

\begin{itemize}
\item [(i)] Let $(\varphi_1, \varrho_1),(\varphi_2, \varrho_2):(G'',\zeta'')_{\xi''}   \rightarrow  (G,\zeta)_\xi $ be morphisms of soft topological groups such that
\[
(\varphi, \varrho) \circ (\varphi_1, \varrho_1) = (\varphi, \varrho) \circ (\varphi_2, \varrho_2).
\]
In this case,
\[\varphi\circ \varphi_1 =\varphi \circ\varphi_2 \quad \text{and} \quad \varrho \circ \varrho_1 = \varrho \circ \varrho_2.\]
Since $\varphi$ is an injective function and $\varrho$ is an injective group homomorphism, it follows that
\(\varphi_1 = \varphi_2 \quad \text{and} \quad \varrho_1 = \varrho_2.\)
Therefore, $(\varphi_1, \varrho_1) = (\varphi_2, \varrho_2)$, which shows that $(\varphi, \varrho)$ is a monomorphism.
\item[(ii)] Let $(\varphi_3,\varrho_3),(\varphi_4,\varrho_4):(G',\zeta')_{\xi'}  \rightarrow  (G''',\zeta''')_{\xi'''}  $ be morphisms of soft topological groups such that \[(\varphi_3, \varrho_3) \circ (\varphi, f) = (\varphi_4, \varrho_4) \circ (\varphi, \varrho).\] Then, \[(\varphi_3 \circ \varphi, \, \varrho_3 \circ \varrho) = (\varphi_4 \circ \varphi, \ \varrho_4 \circ \varrho),
\]
which implies \(
\varphi_3 \circ \varphi = \varphi_4\circ \varphi \quad \text{and} \quad \varrho_3 \circ \varrho = \varrho_4 \circ \varrho.\)
Since $\varphi$ is a surjective function and $\varrho$ is a surjective group homomorphism, it follows that \(\varphi_3 = \varphi_4\quad \text{and} \quad \varrho_3 = \varrho_4.\)
Therefore, $(\varphi_3, \varrho_3) = (\varphi_4, \varrho_4)$, which shows that $(\varphi, \varrho)$ is an epimorphism.
\end{itemize}
\end{proof}
\begin{thm}
Let $(G,\zeta)_\xi $ and $(G',\zeta')_{\xi'} $ be soft topological groups, and let
\[
(\varphi,\varrho)\colon (G,\zeta)_\xi\to (G',\zeta')_{\xi'}
\]
be a morphism of soft topological groups. If $(K,\zeta_K)_\xi$ is a soft topological subgroup of $(G,\zeta)_\xi$, then
$ (\varrho(K),\zeta_{\varrho(K)})_{\xi'} $
is a soft topological subgroup of $(G',\zeta')_{\xi'}$.
\end{thm}

\begin{proof}
 Since  $\varrho$ is a homomorphism and $K$ is a subgroup of  $G$, the image $\varrho(K)$ is a subgroup of $G'$.

 By Proposition \ref{softsubgrp}, since $\varrho(K)$ is a subgroup of $G'$,
$ (\varrho(K),\zeta_{\varrho(K)})_{\xi'} $ is a soft topological subgroup.
\end{proof}
\begin{prop}
Let  $(\varphi, \varrho) : (G,\zeta)_\xi  \to ( G',\zeta')_{\xi'}$  be a  morphism of soft topological groups.
If $(\varphi, \varrho)$ is a monomorphism, then $\varrho$ is injective.
\end{prop}
\begin{proof}
Let $\mathbb{I}_\xi : \xi \to \xi$ be the identity function, $K = \ker(\varrho)$,  and $i : K \to G$ be the inclusion group homomorphism.  And assume that $\zeta_K$ is the soft subtopology on $K$ induced from $\zeta$.
Then
\[
(\mathbb{I}_\xi, i) : (K,\zeta_K)_\xi\rightarrow (G,\zeta)_\xi
\]
is a morphism of soft topological groups. Here, it is clear from Proposition \ref{restricted} that $(\mathbb{I}_\xi, i)$ is soft continuous.

Let  \(\mathbb{I}_\xi : \xi \to \xi\) be the identity function and \(\varrho_1 : K \to G\) be the trivial group homomorphism.
Now, let us show that
\[(\mathbb{I}_\xi, \varrho_1) :  (K,\zeta_K)_\xi\rightarrow (G,\zeta)_\xi\]
is soft continuous.
 Let $x \in K$, and   $U'$ be any soft open neighborhood of  $\varrho_1(x)=1_G$. Then    $U'_K  $ is a soft open neighborhood of  $x$  in  $ \zeta_K$  such that  $(\mathbb{I}_\xi, \varrho_1)(U'_K) \widetilde{\sqsubseteq} U'$.
Therefore,$(\mathbb{I}_\xi, \varrho_1)$ is soft continuous, and
\[
\begin{tikzcd}
(K,\zeta_K)_\xi \ar[r,shift left=.75ex,"{(\mathbb{I}_\xi, i)}"]
  \ar[r,shift right=.75ex,swap,"{(\mathbb{I}_\xi, \varrho_1 )}"]
&(G,\zeta)_\xi  \ar[r,"{(\varphi, \varrho)}"]
& ( G',\zeta')_{\xi'} .
\end{tikzcd}
\]
Since
\[
(\varphi, \varrho) \circ (\mathbb{I}_{\xi}, i) = (\varphi, \varrho) \circ (\mathbb{I}_{\xi}, \varrho_1)
\]
and \((\varphi, \varrho)\) is monic, it follows that
\[
(\mathbb{I}_{\xi}, i) = (\mathbb{I}_{\xi}, \varrho_1),
\]
and hence
\[
i =\varrho_1.
\]
Finally, let us show that \( \varrho \) is injective. Let \( x, y \in K = \ker(\varrho) \), and suppose that \( \varrho(x) = \varrho(y)=1_G \). Then

\[
x=i(x) = \varrho_1(x)=\varrho_1(y)=i(y)=y.
\]
From this it follows that \( x = y \), so \( K = \ker(\varrho) = \{ 1_G \} \) is a singleton. Therefore, \( \varrho \) is injective.
\end{proof}

\begin{prop}
Let   $(\varphi,\varrho) : (G,\zeta)_\xi  \to ( G',\zeta')_{\xi'}$
  be a morphism of soft topological groups.
If $(\varphi, \varrho)$ is an epimorphism, then $\varrho$ is surjective.
\end{prop}

\begin{proof}
Let \( K = \operatorname{Im}(f) \). Let
\(\varrho_1: H \to H/K\) be the canonical epimorphism and let
\(\varrho_2: H \to H/K\) be the trivial group homomorphism. Define the soft indiscrete topology on \( H/K \) as
\(\zeta'' = \{ \emptyset_{\xi'}, {H/K}_{\xi'} \}.\)
Then
\[
(\mathbb{I}_\xi, \varrho_1),(\mathbb{I}_\xi, \varrho_2):( G',\zeta')_{\xi'} \to ( H/K, \zeta'')_{\xi'} \]
 are morphisms of soft topological groups, and

\[
\begin{tikzcd}
(G,\zeta)_\xi   \ar[r,"{(\varphi, \varrho)}"]& (G',\zeta')_{\xi'}  \ar[r,shift left=.75ex,"{(\mathbb{I}_\xi, \varrho_1)}"]
  \ar[r,shift right=.75ex,swap,"{(\mathbb{I}_\xi, \varrho_2 )}"]
& ( H/K, \zeta'')_{\xi'}.
\end{tikzcd}
\]
If \((\mathbb{I}_{\xi}, \varrho_1) \circ (\varphi, \varrho) = (\mathbb{I}_{\xi}, \varrho_2) \circ (\varphi, \varrho) \), then \[(\mathbb{I}_{\xi}\circ \varphi, \varrho_1\circ \varrho)  = (\mathbb{I}_{\xi} \circ \varphi,  \varrho_2\circ \varrho),\] and thus
 \begin{equation} \label{epic2}
\varrho_1\circ \varrho = \varrho_2\circ \varrho.
\end{equation}

Moreover, since
 \[
(\mathbb{I}_{\xi}, \varrho_1) \circ (\varphi, \varrho) = (\mathbb{I}_{\xi}, \varrho_2) \circ (\varphi, \varrho)
\] and  \((\varphi, \varrho)\) is epic, it follows that  \[
(\mathbb{I}_{\xi}, \varrho_1)  = (\mathbb{I}_{\xi}, \varrho_2).
\]
and hence
\begin{equation} \label{epic1}
\varrho_1 = \varrho_2.
\end{equation}
From \eqref{epic2}  and \eqref{epic1}, it follows that \( \varrho \) is a group homomorphism which is an epimorphism.
Since every group epimorphism is surjective, \( \varrho \) is surjective.
\end{proof}

\begin{prop}
Let \( (G_1,\zeta_1 )_{\xi_1}  \) and \( (G_2,\zeta_2 )_{\xi_2}\) be soft topological groups.
Let \( \zeta_\times \) be the soft product topology induced from the collection \(\{ W_1 \widetilde{\times} W_2 \mid W_i \in \zeta_i,\ i = 1, 2\}.\)
Then
\[(G_1 \times G_2, \zeta_\times)_ {\xi_1 \times \xi_2}\]
is a soft topological group where \( G_1 \times G_2 \) is the direct product group and \( \xi_1 \times \xi_2 \) is the Cartesian product of the parameter sets.
\end{prop}
\begin{proof}
    Since  \(( G_1,\zeta_1 )_{\xi_1}  \) is a soft topological group, for any \( a_1, b_1 \in G_1 \) and any soft open set \( W_1 \in \zeta_1 \) such that \( a_1b_1^{-1}  \widetilde{\in}W_1 \),
there are soft open sets \( U_1, V_1 \in \zeta_1 \) such that
\[a_1  \widetilde{\in}U_1, \quad b_1  \widetilde{\in} V_1, \quad \text{and} \quad U_1 (V_1)^{-1}  \widetilde{\sqsubseteq}W_1.\]

Similarly, since \( (G_2,\zeta_2 )_{\xi_2}\) is a soft topological group,
for any \( a_2, b_2 \in G_2 \) and any soft open set \( W_2 \in \zeta_2 \) such that \( a_2 b_2^{-1}  \widetilde{\in}W_2 \),
there exist soft open sets \( U_2, V_2 \in \zeta_2 \) such that
\[a_2  \widetilde{\in}U_2, \quad b_2  \widetilde{\in}V_2, \quad \text{and} \quad U_2 (V_2)^{-1} \widetilde{\sqsubseteq }W_2.\]

\( (a_1, a_2), (b_1, b_2) \in G_1 \times G_2 \).
Then \( (a_1 b_1^{-1}, a_2 b_2^{-1}) \in G_1 \times G_2 \).
Moreover, since \(a_1 \widetilde{\in}U_1, \quad b_1  \widetilde{\in}V_1, \quad a_2  \widetilde{\in}U_2, \quad b_2  \widetilde{\in}V_2,\)
it follows that
\[(U_1 (V_1)^{-1}) \widetilde{\times} (U_2 (V_2)^{-1})  \widetilde{\sqsubseteq}W_1\widetilde{\times} W_2.\]
This shows that
\[
(G_1 \times G_2, \zeta_\times)_ {\xi_1 \times \xi_2}
\]
is a soft topological group.

\end{proof}
The soft topological group
$(G_1 \times G_2, \zeta_\times)_ {\xi_1 \times \xi_2}$
is called the \textbf{soft product topological group}.
\begin{thm}\label{projectionmap}
 Let \(( G_1,\zeta_1 )_{\xi_1}  \) and \( (G_2,\zeta_2 )_{\xi_2}\) be two soft topological groups,
and let $(G_1 \times G_2, \zeta_\times)_ {\xi_1 \times \xi_2}$
be the  soft product topological group. Define  projection maps:
\[P_{\xi_1} : \xi_1 \times \xi_2 \to \xi_1, \quad q_{G_1} : G_1 \times G_2 \to G_1,\]
\[P_{\xi_2} : \xi_1 \times \xi_2 \to \xi_2, \quad q_{G_2} : G_1 \times G_2 \to G_2.\]
Then the pairs
\[
(P_{\xi_1}, q_{G_1}) : (G_1 \times G_2, \zeta_\times)_ {\xi_1 \times \xi_2} \to (G_1,\zeta_1 )_{\xi_1}
\]
and
\[
(P_{\xi_2}, q_{G_2}) : (G_1 \times G_2, \zeta_\times)_ {\xi_1 \times \xi_2} \to  (G_2,\zeta_2 )_{\xi_2}
\]
are morphisms of soft topological groups.

\end{thm}
\begin{proof}
 It is known that the maps
\(
q_{G_1} : G_1 \times G_2 \to G_1 \quad \text{and} \quad q_{G_2} : G_1 \times G_2 \to G_2
\)
are group homomorphisms. Let us show that the morphism
\[
(P_{\xi_1}, q_{G_1}) : (G_1 \times G_2, \zeta_\times)_ {\xi_1 \times \xi_2} \to (G_1,\zeta_1 )_{\xi_1}
\]
is soft continuous.

Let \( (x, y) \in G_1 \times G_2 \). Since
\(
q_{G_1}(x, y) = x,
\)
for every soft neighborhood $U_1$ of $q_{G_1}(x,y) = x$, we have
\[
x=q_{G_1}(x,y) \widetilde{\in} (P_{\xi_1},q_{G_1})(U_1 \widetilde{\times} V_1) =U_1  \widetilde{\sqsubseteq}U_1.
\]
So, $(P_{\xi_1}, q_{G_1})$ is soft continuous.
Similarly, one can show that
\[
(P_{\xi_2}, q_{G_2}) : \langle G_1 \times G_2, \tau, \xi_1 \times \xi_2 \rangle \to \langle G_2, \tau_2, \xi_2 \rangle
\]
is also soft continuous.
\end{proof}
\begin{thm}
 Let \( (G_1,\zeta_1 )_{\xi_1}  \) and \( (G_2,\zeta_2 )_{\xi_2}\) be any two objects in $\mathbf{STGrp}$. Then, the product of \( (G_1,\zeta_1 )_{\xi_1}  \) and \( (G_2,\zeta_2 )_{\xi_2}\) in $\mathbf{STGrp}$ is $$((G_1 \times G_2, \zeta_\times)_ {\xi_1 \times \xi_2}, (P_{\xi_i}, q_{G_i})),$$ for $i=1,2$.
\end{thm}

\begin{proof}
 Let $ (H, \sigma)_\xi $ be an object in $\mathbf{STGrp}$, and let $(\varphi_1, \varrho_1): (H, \sigma)_\xi\to (G_1,\zeta_1 )_{\xi_1}$ and $(\varphi_2, \varrho_2): (H, \sigma)_\xi \to (G_2,\zeta_2 )_{\xi_2}$ be morphisms of soft topological groups. Consider the maps
\(
(\varphi_1, \varphi_2) : \xi \to \xi_1 \times \xi_2,
\)
defined by
\(
(\varphi_1, \varphi_2)(e) = (\varphi_1(e), \varphi_2(e))
\), which is a function, and \(
(\varrho_1, \varrho_2) : H \to G_1 \times G_2,
\)
defined by
\(
(\varrho_1, \varrho_2)(h) = (\varrho_1(h), \varrho_2(h))
\), which is a group homomorphism.
Now we can define the soft mapping
\[
(\vartheta, \chi) = \big( (\varphi_1, \varphi_2), (\varrho_1, \varrho_2) \big) :( H, \sigma)_ \xi \to ( G_1 \times G_2, \zeta_\times)_ {\xi_1 \times \xi_2}
\] by using the maps $(\varphi_1, \varphi_2)$ and $(\varrho_1, \varrho_2)$.
Since $(\varphi_1, \varrho_1)$ and $(\varphi_2, \varrho_2)$ are soft continuous, for every \( h \in H \) and for every soft neighborhood \( W_1 \) of \( \varrho_1(h) \), there is a soft neighborhood \( U_1 \) such that
\[
\varrho_1(h)  \widetilde{\in}(\varphi_1, \varrho_1)(U_1)  \widetilde{\sqsubseteq}W_1.
\]

Similarly, for every \( h \in H \) and every soft neighborhood \( W_2 \) of \( \varrho_2(h) \), there is a soft neighborhood \( U_2 \) such that
\[
\varrho_2(h)  \widetilde{\in}(\varphi_2, \varrho_2)(U_2) \widetilde{\sqsubseteq }W_2.
\]

Hence, for every \( h \in H \) and every soft neighborhood \( (U_1\widetilde{\times} U_2) \) of
\(
\chi(h) = (\varrho_1, \varrho_2)(h) = (\varrho_1(h), \varrho_2(h)),
\)
we have
\[
(\varrho_1, \varrho_2)(h)  \widetilde{\in}((\varphi_1, \varrho_1), (\varphi_2, \varrho_2))(U_1\widetilde{\times} U_2)  \widetilde{\sqsubseteq } W_1\widetilde{\times} W_2.
\]

Therefore, the map \( (\vartheta, \chi) \) is soft continuous, which shows the existence of the soft topological group morphism \( (\vartheta, \chi) \).

Let \( (\gamma, \theta): ( H, \sigma)_ \xi \to ( G_1 \times G_2, \zeta_\times)_ {\xi_1 \times \xi_2} \)
be another soft topological group morphism such that
\[
(P_{\xi_1}, q_{G_1}) \circ (\gamma, \theta) = (\varphi_1, \varrho_1) \quad \text{and} \quad (P_{\xi_2}, q_{G_2}) \circ (\gamma, \theta) = (\varphi_2, \varrho_2).
\]
Since
\(
P_{\xi_1} \circ \gamma = P_{\xi_1} \circ \vartheta,
\)
it follows that \( \gamma = \vartheta \), and since
\(
q_{G_1} \circ \theta = q_{G_1} \circ \chi,
\)
it follows that \( \theta = \chi \).
Therefore, \( (\vartheta, \chi) \) is unique.

\[
\begin{tikzpicture}
  \node (s) {$( H, \sigma)_ \xi$};
  \node (xy) [below=2 of s] {$( G_1 \times G_2, \zeta_\times)_ {\xi_1 \times \xi_2}  $};
  \node (x) [left=2 of xy] {$( (G_1,\zeta_1 )_{\xi_1}$};
  \node (y) [right=2 of xy] {$( (G_2,\zeta_2 )_{\xi_2}$};
  \draw[->] (s) to node [sloped, above] {$(\varphi_1,\varrho_1)$} (y);
  \draw[<-] (x) to node [sloped, above] {$(\varphi_2,\varrho_2)$} (s);
  \draw[->, dashed] (s) to node {$(\vartheta, \chi)$} (xy);
  \draw[->] (xy) to node [below] {$(P_{\xi_1}, q_{G_1})$} (x);
  \draw[->] (xy) to node [below] {$(P_{\xi_2}, q_{G_2})$} (y);
\end{tikzpicture}
\]

\end{proof}

\begin{thm}
    Let \( \{1_G\} \) be the trivial group and \( \{e\} \) be a singleton set.
Consider the soft indiscrete topology \( \zeta \) on \( \{1_G\} \).
Then the triple \( ( \{1_G\}, \zeta)_{\{e\}} \) is the terminal object in the category of soft topological groups  \textbf{STGrp}.
\end{thm}
\begin{proof}
    Let \( ( H, \sigma)_ \xi  \) be an arbitrary soft topological group. And let \( \varrho : H \to \{1_G\} \) be the trivial group homomorphism and \( \varphi : \xi \to \{e_1\} \) be the constant function.
Then there exists a unique morphism
\[
( H, \sigma)_ \xi  \xrightarrow{(\varphi, \varrho)}( \{1_G\}, \zeta)_{\{e\}}
\]
in the category \textbf{STGrp}.  Here, since \( \zeta \) is the soft indiscrete topology, the morphism \( (\varphi, \varrho) \) is soft continuous.
Therefore, \(  \{1_G\}, \zeta)_{\{e\}} \) is the terminal object in the category \textbf{STGrp}.
\end{proof}

\begin{cor}
  The category  \textbf{STGrp} is a symmetric monoidal category.
\end{cor}

\section*{Conflict of interest}
The authors declare that they have no conflicts of interest.

\section*{Author contributions}
Nazmiye Alemdar: Conceptualization, formal analysis, investigation,
methodology,  validation, writing-original
draft, writing-review \& editing. H\"{u}rmet Fulya Ak{\i}z: Conceptualization, methodology,  writing-original
draft. Halim Ayaz: Conceptualization, investigation. The final version of the manuscript was read and approved by all
authors.







\end{document}